\documentclass[UTF8]{article}
\usepackage{graphicx}
\usepackage{subfigure}
\usepackage{algorithm}
\usepackage{float,inputenc}
\usepackage{diagbox}
\usepackage{booktabs}
\usepackage{algorithmicx}
\usepackage{amsmath,amssymb,amsthm,enumerate,mathrsfs}
\usepackage{caption}
\usepackage[noend]{algpseudocode}
\usepackage[justification=centering]{caption}
\usepackage{dsfont}
\usepackage{setspace}
\usepackage{tabularx}
\newtheorem{theorem}{Theorem}[section]

\theoremstyle{definition}
\newtheorem{definition}[theorem]{Definition}

\newtheorem{remark}[theorem]{Remark}
\newcommand{\upcite}[1]{\textsuperscript{\textsuperscript{\cite{#1}}}}

\def\cX{{\cal X}}

\numberwithin{equation}{section}

\begin{document}
\makeatletter

\title{\Large {\bf A Modification Piecewise Convexification Method for Box-Constrained Non-Convex Optimization Programs}}

\author{Qiao Zhu, Liping Tang\thanks{Corresponding Author}, and Xinmin Yang}


\maketitle
\vspace*{0mm}

\vspace{2mm}

\footnotesize{
\noindent\begin{minipage}{14cm}
{\bf Abstract:}
This paper presents a piecewise convexification method to approximate the whole approximate optimal solution set of non-convex optimization problems with box constraints. In the process of box division, we first classify the sub-boxes and only continue to divide only some sub-boxes in the subsequent division. At the same time, applying the $\alpha$-based Branch-and-Bound ({\rm$\alpha$BB}) method, we construct a series of piecewise convex relax sub-problems, which are collectively called the piecewise convexification problem of the original problem. Then, we define the (approximate) solution set of the piecewise convexification problem based on the classification result of sub-boxes. Subsequently, we derive that these sets can be used to approximate the global solution set with a predefined quality. Finally,  a piecewise convexification algorithm with a new selection rule of sub-box for the division and two new termination tests is  proposed.  Several instances verify that these techniques are beneficial to improve the performance of the algorithm.
\end{minipage}
 \\[5mm]


\hbox to14cm{\hrulefill}\par


\section{Introduction}

Non-convex optimization problems arise frequently in machine learning  \cite{zhong2012training,jain2017non,wen2018survey}, and in many other applications. Meanwhile, how to effectively solve non-convex optimization problems has gained much attention. So far, the majority of global approximation algorithms are designed, see, \cite{liuzzi2019new,locatelli2021global,schoen2021efficient,yang2018successive,marmin2019globally,sun2001convexification,wu2005convexification,xia2020survey}. In particular,   the ${\rm\alpha BB}$ method has become an increasingly important global optimization method in the design of efficient and computationally tractable numerical algorithms for non-convex optimization problems, see, \cite{androulakis1995alphabb,Milan2014On,hladik2016extension,kazazakis2018arbitrarily}.
It is worth noticing that these algorithms aim in general at determining a single globally optimal solution, and the majority of application problems may be existed with many or even infinite of many  globally optimal solutions. To the best of our knowledge, the number of algorithms for determining representations of the set of approximate globally optimal solutions seems to be still quite limited. In \cite{eichfelder2016modification}, Eichfelder and Gerlach generalize the classical ${\rm\alpha BB}$ method to find a representation of the whole optimal solution set with predefined quality for a non-convex optimization problem. As they pointed out, however, some additional variables and the additional while loop are necessary.


Motivated by \cite{eichfelder2016modification}, we develop a piecewise convexification method to approximate the whole global optimal solution set of non-convex optimization problems.
Incorporating the ${\rm\alpha BB}$ method and the interval subdivision technique, we firstly piecewise relax the original problem and obtain a series of convex relaxation sub-problems, collectively called the piecewise convexification problem  of the original problem. Then, we construct the (approximate) solution sets of the piecewise convexification problem by comparing all the (approximate) optimal solutions of the convex relaxation sub-problems on each sub-box, and we show that these constructed sets can be used to approximate the global optimal solution set with a predefined quality. Finally, a new piecewise convexification algorithm is proposed, which incorporates a new selection rule of sub-box for the division and two new termination rules. Furthermore, several instances verify that these selection rule and termination tests are conducive to improving the effectiveness and the speed of the algorithm.
%

This paper is organized as follows. Section 2 summarizes some basic definitions of the optimization problem. It also introduces the $\alpha{\rm BB}$ method and the interval division. In Section 3, we firstly propose a piecewise convexification method for the non-convex optimization problem and then analyze the solution set of this piecewise convexification optimization problem. More importantly, some relationships  between the (approximate) optimal solution set of the convexification problem and the original optimization problem are also stated in detail. A new algorithm that generates the subset of approximation global solutions is presented in Section 4. Finally, we report and discuss several numerical experiments in section 5.

\section{Preliminaries}
Let $\mathbb{I}$ denoted the set of all real nonempty closed boxes and $\mathbb{I}^{n}$ denote the set of all n-dimensional boxes.
For a given  box ${X}\in \mathbb{I}^{n}$, we set ${X}=[a,b]:=\prod\limits_{i=1}^{n}[a_{i},b_{i}]$ where ${a}=(a_{1},\cdots,a_{n})$ and ${b}=(b_{1},\cdots,b_{n})$. Thus,  $x\in {X}$ denotes $x_{i}\in[a_{i},b_{i}]$ for any $i\in\{1,\cdots,n\}$.
In this paper, let ${X}=[{a},b]$, and then we consider the following non-convex optimization problem(NCOP):
\begin{align*}
{\text{(NCOP)}}~~\min\limits_{x\in {X}} f(x),
\end{align*}
where $f:\mathbb{R}^{n}\rightarrow\mathbb{R}$ is a non-convex  twice continuously differentiable function. We start with an overview of the (approximation) optimal solution of (NCOP).
\begin{definition}(see \cite{eichfelder2016modification})
Let $\varepsilon>0$, $X$ be a nonempty subset of $\mathbb{R}^{n}$, and $f:X\rightarrow\mathbb{R}$ such that $\arg\min\limits_{x\in X}f(x)\neq \varnothing$.

(a) A point $\tilde{x}\in X$ is an optimal solution of $f$ w.r.t. $X$, if
\begin{align*}
f(\tilde{x})\leq f(y) ~~\forall y\in X.
\end{align*}
$\cX_{op}:=\{x\in X: f(x)\leq f(y), ~\forall y\in X\}$ is the optimal solution set.

(b) A point $\bar{x}\in X$ is an $\varepsilon$-minimal point of $f$ w.r.t. $X$, if
\begin{align*}
f(\bar{x})\leq f(y)+\varepsilon ~~\forall y\in X.
\end{align*}
$\cX_{op}^{\varepsilon}:=\{x\in X: f(x)\leq f(y)+\varepsilon, \forall y\in X\}$ is the approximate solution set.
\end{definition}
Before proceeding further, we hereby give brief descriptions of the ${\rm \alpha BB}$ method and the interval subdivision, because they will play an important role in designing the piecewise convexification method.

{\bf \subsection{The ${\rm{\alpha BB}}$ Method}}
For solving non-convex problems in global optimization, the $\alpha{\rm BB}$ method constructs a convex relaxation estimation function of $f$ w.r.t. $X$, see   \cite{eichfelder2016modification,adjiman1998globalI,adjiman1998globalII,androulakis1995alphabb}. More precisely, Let $f: X\rightarrow \mathbb{R}$ be a real-valued twice continuously differentiable function and  $X=[a,b]\in\mathbb{I}^{n}$. A convex lower relaxation function $F_{X}^{\alpha}: X\rightarrow \mathbb{R}$ of $f$ by the idea of the $\alpha$BB method is defined in \cite{adjiman1998globalI} as follows:
\begin{align}\label{F-alpha:2021051501}
F_{X}^{\alpha}(x)=f(x)+\sum\limits_{i=1}^{n}\alpha_{i}(a_{i}-x_{i})(b_{i}-x_{i}),
\end{align}
where parameter $\ \alpha:=(\alpha_{1},\cdots,\alpha_{n})$ guarantees the convexity of $F_{X}^{\alpha}$ on $X$.

For estimating the value of $\alpha$ several methods already  exist, see,  \cite{adjiman1998globalI,nerantzis2019tighter, skjal2014new}. In this article, we directly adopt the following method from \cite{adjiman1998globalI} to roughly calculate the value $\alpha_{i}(i=1,\cdots,n)$,
as defined by
\begin{align}\label{alpha:2021051201}
\alpha_{i}:=\max\left\{0,-\frac{1}{2}\left(\min\limits_{x\in X}\nabla^2f(x)_{ii}-\sum(X,i,d)\right)\right\},~i = 1,\cdots,n
\end{align}
where $d:=b-a\in\mathbb{R}^{n}$, Hessian matrix $\nabla^2f(x)=\Big(\nabla^2f(x)_{ij}\Big)$ and
\begin{align}\label{sum:2021051701}
\sum(X,i,d):= \sum\limits_{i\neq j}\max\left\{\left|\min\limits_{x\in X}\nabla^2f(x)_{ij}\right|,\left|\max\limits_{x\in X}\nabla^2f(x)_{ij}\right|\right\}\frac{d_{j}}{d_{i}}.
\end{align}
Noting that $\alpha_{i}$ is a finite value.
Let $\mathbf{1}=(1,\cdots,1)$,  a lower bound of the minimum eigenvalue of $\nabla^2f(x)$ w.r.t. $X$ is given in \cite{adjiman1998globalI}, i.e., $$\lambda_{\min}^{X}\geq \min\limits_{i}\Big(\min\limits_{x\in X}\nabla^2f(x)_{ii}-\sum(X,i,\mathbf{1})\Big).$$
It is well known that if the lower bound
\begin{align}\label{lambda:2021051702}
\tilde{\lambda}_{X}:=\min\limits_{i}\Big(\min\limits_{x\in X}\nabla^2f(x)_{ii}-\sum(X,i,\mathbf{1})\Big)\geq 0,
\end{align}
then $f$ is obviously convex on $X$, not vice versa. In addition, for any boxes $X_{1}$ and $X_{2}$ with $X_1\subseteq X_{2}$, it is clear that $\tilde{\lambda}_{X_2}\leq\tilde{\lambda}_{X_1}$.

Moreover, in \cite{nerantzis2019tighter}, the maximum separation distance between $f$ and $F_{X}^{\alpha}$ over $X$ is typically of the form
\begin{align}\label{eq:2021052601}
D(X)=\max\limits_{x\in X}\|f(x)-F_{X}^{\alpha}(x)\|=\sum\limits_{i=1}^{n}\alpha_{i}\Big(\frac{b_{i}-a_{i}}{2}\Big)^2,
\end{align}
which shows that $D(X)$ is determined by the interval $[a,b]$ and $\alpha$.

{\bf\subsection{Interval Division}}
\noindent
As shown in Eq.(\ref{eq:2021052601}), a smaller interval helps to generate a tighter under-estimator of the original function. Thus, we try to divide the whole box into some sub-boxes in order to better approximate the original function.

In this paper, let $\mathbf{Y}^t:= \{Y^{1}, Y^{2}, \cdots, Y^{M_t}\}$ be a subdivision of $X$ with respect to the number of division $t\in\mathbb{N}$, which satisfies that
\begin{align*}
X = \bigcup\limits_{{k_t}=1}^{M_{t}}Y^{k_t} \text{~and~} \mu^{n}(Y^{k_t}\cap Y^{j_t})=0 ~\forall k_t, j_t\in\{1,\cdots,M_{t}\},
\end{align*}
where $\mu^{n}$ denotes the Lebesgue measure on $\mathbb{R}^{n}$ and the subinterval $Y^{k_t}$ abbreviate  as $Y^{k_t}=[a^{k_t},b^{k_t}]=\prod\limits_{i=1}^{n}[a_{i}^{k_t}, b_{i}^{k_t}]$. It is worth noting that the construction of the subdivision $\mathbf{Y}^{(t+1)}$ of $X$ w.r.t $(t+1)$ is based on $\mathbf{Y}^{t}$, which is to select one or more sub-intervals from $\mathbf{Y}^{t}$ to divide. In what follows, we introduce the interval division method of any subinterval $Y^{k_t}$.

For a given box $Y^{k_t}=[a^{k_t}, b^{k_t}]\in\mathbb{I}^{n}$, the branching index $l$ is defined by
\begin{align*}
l:=\min\{j\in\{1,\cdots,n\}:j\in\arg\max\limits_{j\in\{1,\cdots,n\}}(b^{k_t}_{j}-a^{k_t}_{j})\},
\end{align*} and $Y^{k_t}$ splits into two subsets $Y^{k_t,1}$ and $Y^{k_t,2}$ based on direction $l$   by
\begin{align*}
&Y^{k_t, 1}:=\prod\limits_{i=1,i\neq l}^{n}[a^{k_t}_{i}, b^{k_t}_{i}]\times \left[a^{k_t}_{l}, \frac{a^{k_t}_{l}+b^{k_t}_{l}}{2}\right],\\
&Y^{k_t, 2}:=\prod\limits_{i=1,i\neq l}^{n}[a^{k_t}_{i}, b^{k_t}_{i}]\times \left[\frac{a^{k_t}_{l}+b^{k_t}_{l}}{2}, b^{k_t}_{l}\right].
\end{align*}
Clearly, $Y^{k_t}=Y^{k_t,1}\cup Y^{k_t,2}$ and $Y^{k_t, 1}, Y^{k_t, 2}\in \mathbf{Y}^{(t+1)}$. For simplicity, we may define the splitting operator $\text{Sp}(Y^{k_t})=\{Y^{k_t,1}, Y^{k_t,2}\}$. In this paper, we define the length of the subdivision $\mathbf{Y}^{t}$ of $X$ by
\begin{align*}
|T(\mathbf{Y^{t}})|:=\max\limits_{k_{t}\in\{1,\cdots,M_{t}\}}\{\|a^{k_t}-b^{k_t}\|_{2}^{2}\}=\max\limits_{k_t\in\{1,\cdots,M_{t}\}}\left\{\sum\limits_{i=1}^{n}(b_{i}^{k_t}-a_{i}^{k_t})^2\right\}.
\end{align*}
\begin{remark}\label{remark:20210715001}
Let $x^{*}\in\cX_{op}$.  For any given subdivision $\mathbf{Y}^{t}$ of $X$, there exists $Y^{k_t}\in\mathbf{Y}^{t}$ such that $x^{*}\in Y^{k_t}$ and $x^{*}$ is an optimal solution of $f$ w.r.t. $Y^{k_t}$.
\end{remark}

\section{Piecewise Convexification Method  for (NCOP)}
In this section, we firstly introduce the piecewise convexification problem for the non-convex optimization problem (PC-NCOP). The solution sets of the piecewise  convexification optimization problem are constructed and discussed in detail. Finally, we analyze some relationships between the solution sets of the piecewise convexification optimization problem and the (approximate) globally optimal solution set of the original non-convex optimization  problem.

{\bf\subsection{Piecewise Convexification Problem}}
In order to approximate the global solution set of a non-convex optimization problem,  we use the interval subdivision method to divide $X$ into several sub-boxes and use the $\alpha$BB method to relax this problem on each sub-box of $X$ rather than $X$. Thus it
is referred to as the piecewise convexification method. In the following, we discuss this method in detail.

Let $\mathbf{Y}^{t}:= \{Y^{1}, Y^{2},\cdots, Y^{M_t}\}$ be a subdivision of $X$.
Then we consider the same convex relaxation subproblem on $Y^{k_t}=[a^{k_t},b^{k_t}]\in \mathbf{Y}^{t}$ as \cite{adjiman1998globalI}, that is,
\begin{align}\label{alphaBB}
\min\limits_{x\in Y^{k_t}} F_{ k_t}^{\alpha^{k_t}}(x):=f(x)+\sum\limits_{i=1}^{n}{\alpha_{i}^{k_t}}(a_{i}^{k_t}-x_{i})(b_{i}^{{k_t}}-x_{i}),
\end{align}
where 
%
if $\tilde{\lambda}_{Y^{k_t}}\geq 0$ estimated by (\ref{lambda:2021051702}), then $\alpha_{i}^{k_t}=0$. Otherwise, $\alpha_{i}^{k_t}$ is computed by (\ref{alpha:2021051201}) for any $\{1,\cdots,n\}$.  Let $\alpha^{k_t}:=(\alpha_{1}^{k_t},\cdots,\alpha_{n}^{k_t})$. Apparently, $F_{k_t}^{\alpha^{k_t}}$ is a convex lower estimation function of $f(x)$ on $Y^{k_t}$ and $F_{k_t}^{\alpha^{k_t}}(x)=f(x)$ for any $x\in Y^{k_t}$ when $\tilde{\lambda}_{Y^{k_t}}\geq 0$.
Let $\cX_{ap}^{k_t}$ and $\cX_{ap}^{k_t,\varepsilon}$ denote the set of all optimal solutions and the set of all approximate solutions of (\ref{alphaBB}), respectively, that is,
\begin{align}
&\cX_{ap}^{k_t}:=\{x\in Y^{k_t}: F_{k_t}^{\alpha^{k_t}}(x)\leq F_{k_t}^{\alpha^{k_t}}(y) \text{~for~any~} y\in Y^{k_t}\},\label{X_ap}\\
&\cX_{ap}^{k_t,\varepsilon}:=\{x\in Y^{k_t}: F_{k_t}^{\alpha^{k_t}}(x)\leq F_{k_t}^{\alpha^{k_t}}(y)+\varepsilon \text{~for~any~} y\in Y^{k_t}\},\label{cXk:20210725}
\end{align}
where $\varepsilon>0$. Obviously, $\cX_{ap}^{k_t}$ and $\cX_{ap}^{k_t,\varepsilon}$ are not empty sets.

For any $k_t\in\{1,\cdots,M_t\}$ and $Y^{k_t}\in\mathbf{Y}^{t}$, $F_{k_t}^{\alpha^{k_t}}$ is a convex relaxation sub-problem of (NCOP) w.r.t. $Y^{k_t}$. Obviously, all convex sub-problems compose the piecewise convexification optimization problem of (NCOP) w.r.t. $X$.


{\bf \subsection{The Solution Set of the Piecewise Convexification Problem}}
In this subsection, we construct the solution set of the piecewise convexification problem w.r.t. $X$ for the subdivision $\mathbf{Y}^{t}$. The construction of this solution set is crucial because it relates to the approximation of the global optimal solution set and  directly affects the performance of the algorithm.

Let $M_t$ represent the index set of all sub-boxes from the subdivision $\mathbf{Y}^t$. Note that $X=\bigcup\limits_{k_t\in M_{t}}Y^{k_t}$. As we all know, if $f$ is convex on a current box, then it is also convex on any sub-box of this box. Thus, we will verify that whether $f$ is already convex on the current box before dividing one box in this paper.
Obviously, we can firstly define two auxiliary indicator sets to judge the convexity of $f$ on its corresponding box, i.e.,
\begin{align}\label{M1(t):2021071401}
M_{1}(t):=\left\{k_t\in M_{t}: \tilde{\lambda}_{Y^{k_t}}\geq 0\right\}\text{~and~}
M_{2}(t):=\left\{k_t\in M_{t}: \tilde{\lambda}_{Y^{k_t}}<0\right\},
\end{align}
where $\tilde{\lambda}_{Y^{k_t}}$ is defined by (\ref{lambda:2021051702}). Moreover, $M_{t} = M_{1}(t)\cup M_{2}(t)$. Clearly, $f$ is convex on any subset of $Y^{k_t}$ for any $k_t\in M_{1}(t)$. However, for any $k_t\in M_{2}(t)$ one cannot assert that $f$ must be non-convex on $Y^{k_t}$ because we only obtain  $\tilde{\lambda}_{Y^{k_t}}<0$ rather than $\lambda_{min}^{Y^{k_t}}<0$, that is, the convexity of $f$ is uncertain on $Y^{k_t}$ for any $k_t\in M_2(t)$.  Thus, we need more concerned about the indexes in $M_{2}(t)$ rather than in  $M_{1}(t)$ for any $t$. Our idea of choosing the box to divide is that no subdivision is applied to the box when $f$ is convex on it,  and the box on which $f$ is non-convex is only selected to be divided.

Then, some notations about the union of solution sets corresponding to the above indexes sets, are represented,  which help us to clearly define the solution set of the piecewise convexification problem.
\begin{align}\label{cXap(C,t):2021071402}
\cX_{ap}^{C}(t):=\bigcup\limits_{k_t\in M_{1}(t)}\cX_{ap}^{k_t} \text{~~and~~} \cX_{ap}^{NC}(t):=\bigcup\limits_{k_t\in M_{2}(t)}\cX_{ap}^{k_t}
\end{align}
 where $\cX_{ap}^{k_t}$ is an optimal solution set of the convex relaxation optimization problem (\ref{alphaBB}) w.r.t. $Y^{k_t}$. Final, in this paper we directly define the set of the piecewise convexification problem $\cX_{ap}(t)$ of the following form
\begin{align}\label{eq:2021053101}
{\cX}_{ap}(t)=\left\{x\in\cX_{ap}^{C\cup NC}(t): f(x)\leq f(y) \text{~for~any~} y\in\cX_{ap}^{C\cup NC}(t)\right\}
\end{align}
where $\cX_{ap}^{C\cup NC}(t):= \cX_{ap}^{C}(t)\cup\cX_{ap}^{NC}(t)$.

In what follows, we analyze the solution set $\cX_{ap}(t+1)$ we defined from the perspective of the division process, which helps us better understand the advantages of this definition.
As mentioned above, no subdivision is applied to this box on which $f$ is convex, that is, from the subdivision $t$ to $(t+1)$ we only divide the boxes whose indicators belong to $M_2(t)$, instead of dividing all the boxes corresponding to $M_{t}$. Then, this division process yields to
two new auxiliary indexes sets based on $M_{2}(t)$, as defined by
\begin{align}
&M_{t}^{t+1,C}:=\left\{j_{t}:\tilde{\lambda}_{Y^{j_t}}\geq 0 \text{~where~}Y^{j_{t}}\in\text{Sp}(Y^{k_t})\text{~and~}k_t\in M_{2}(t)\right\},\label{ineq1:2021041607}\\
&M_{t}^{t+1,NC}:=\left\{j_{t}: \tilde{\lambda}_{Y^{j_t}}<0 \text{~where~} Y^{j_{t}}\in\text{Sp}(Y^{k_t})\text{~and~} k_t\in M_{2}(t)\right\}, \label{ineq1:20210416071}
\end{align}
which indicate that we classify the boxes generated in the division process, and put the index of the new box that makes $f$ convex into $M_{t}^{t+1,C}$, otherwise, put it into $M_{t}^{t+1,NC}$. Obviously, these definitions imply that
\begin{align}
&M_1(t+1)=M_1(t)\cup M_{t}^{t+1, C} \text{~and~} M_2(t+1)= M_{t}^{t+1, NC}\label{M1(t+1):2021071401}.
\end{align}
The union of solution sets about (\ref{ineq1:2021041607}) and (\ref{ineq1:20210416071}) are similarly represented by
\begin{align*}
\cX_{ap}^{newC}(t, t+1):=\bigcup\limits_{k_t\in M_{t}^{t+1,C}}\cX_{ap}^{k_t} \text{~~and~~}
\cX_{ap}^{newNC}(t, t+1):=\bigcup\limits_{k_t\in M_{t}^{t+1,NC}}\cX_{ap}^{k_t}.
\end{align*}
Combining this with (\ref{M1(t+1):2021071401}), one can conduct that
\begin{align*}
\cX_{ap}^{C}(t+1)=\cX_{ap}^{C}(t)\cup\cX_{ap}^{newC}(t,t+1) \text{~~and~~} \cX_{ap}^{NC}(t+1)=\cX_{ap}^{newNC}(t, t+1).
\end{align*}
Therefore, $\cX_{ap}(t+1)$ can be equivalently expressed as the following form:
\begin{align*}
\cX_{ap}(t+1):=\left\{x\in\cX_{ap}(t,t+1): f(x)\leq f(y)\text{~for~any~}
 y\in\cX_{ap}(t,t+1)\right\},
\end{align*}
where $\cX_{ap}(t,t+1):=\cX_{ap}^{C}(t)\cup \cX_{ap}^{newC}(t, t+1)\cup\cX_{ap}^{newNC}(t, t+1)$.

This equivalent form directly demonstrates the rationality and advantage of this definition way of $\cX_{ap}(t+1)$. These are summarized in following remark.
\begin{remark}
(i) When $f$ is convex on $X$, it is easy to check that $\cX_{ap}(t)=\cX_{op}$ for any $t$. This implies that the definition of solution set of the piecewise convexification optimization problem is reasonable.

(ii) There is a significant relationship between $\cX_{ap}(t)$ and $\cX_{ap}(t+1)$ since set $\cX_{ap}(t+1)$ uses part of the information of the subdivision $t$, that is, $\cX_{ap}^{C}(t)$. According to the fact that the subdivision $(t+1)$ is always based on the result of subdivision $t$, it follows that this relationship is reasonable.

(iii) From the subdivision $t$ to $t+1$, we
do not consider the boxes that make $f$ convex in the subdivision $t$. Moreover, we directly use $\cX_{ap}^{C}(t)$ from the result of subdivision $t$  to construct the set $\cX_{ap}(t+1)$, instead of solving these sub-problems repeatedly. These techniques may reduce the number of sub-problems to be solved in the piecewise convexification method.
\end{remark}
Next,  we will discuss the relationship between $\cX_{ap}(t)$ and $\cX_{op}$ when $f$ has the piecewise convex property on $X$.
\begin{theorem}
If there exists the subdivision $\mathbf{Y}^{t_{0}}$  of $X$ such that $M_{2}(t_0)=\emptyset$, then $\cX_{op}=\cX_{ap}(t_0)$.
\end{theorem}
\begin{proof}
$M_{2}(t_0)=\emptyset$ implies that $\cX_{ap}^{NC}(t_0)=\emptyset$, $M_{t_0}=M_{1}(t_0)$ and
\begin{align}\label{2021071502}
\cX_{ap}(t_0)=\{x\in\cX_{ap}^{C}(t_0): f(x)\leq f(y)\text{~for~any~}y\in\cX_{ap}^{C}(t_0)\}.
\end{align}
According to $\mathbf{Y}^{t_{0}}$ is a subdivision of $X$, then $X=\bigcup\limits_{k_t\in M_{t_0}}Y^{k_t}$.
In what follows, we will proof $\cX_{ap}(t_0)=\cX_{op}$. The proof is by contradiction.

If $\cX_{op}\nsubseteq \cX_{ap}(t_0)$, then there exists $\hat{x}\in\cX_{op}$ such that $\hat{x}\notin\cX_{ap}(t_0)$.
Based on the definition of $\cX_{ap}(t_0)$ as shown in  (\ref{2021071502}),  it easy to verify that $\hat{x}\notin\cX_{op}$.


Next, we assume that $\cX_{ap}(t_0)\nsubseteq\cX_{op}$, that is, there exists $\hat{x}\in\cX_{ap}(t_0)$ with $\hat{x}\notin\cX_{op}$. Thus,  one can find $\hat{y}\in X$ satisfied $f(\hat{y})<f(\hat{x})$. Since $\mathbf{Y}^{t_{0}}$ is a subdivision, then let $\hat{y}\in Y^{j_{k_0}}$  where $k_{t_0}\in M_{t_0}=M_{1}(t_0)$.  Based on the non-emptiness of $\cX_{ap}^{k_{t_0}}$, there must exist $\hat{z}$ such that $\hat{z}\in\cX_{ap}^{k_{t_0}}$. It follows that $f(\hat{z})\leq f(\hat{y})$, which yields to $f(\hat{z})<f(\hat{x})$ with $\hat{z}\in\cX_{ap}^{k_{t_0}}\subseteq\cX_{ap}^{C}(t_{0})$.  This apparently contradicts the fact that $\hat{x}\in\cX_{ap}(t_{0})$. Thus $\cX_{ap}(t_{0})\subseteq\cX_{op}$.

Obviously, by the above analysis one can conduct $\cX_{ap}(t_{0})=\cX_{op}$.
\end{proof}
This theorem shows that, for the non-convex problem with the piecewise convex properties, i.e., $f$ is non-convex on $X$ and $f$ is convex on each sub-box of $X$ for some subdivision $t_0$, the proposed piecewise convexification method can explore all  global optimal solutions of this non-convex optimization problem.

{\bf \subsection{Approximation of the Global Optimal Solution Set}}
In this subsection, we will show that set $\cX_{ap}(t)$ is actually a lower bound set of $\cX_{op}^{\varepsilon_{t}}$, and in order to obtain the upper bound set of $\cX_{op}^{\varepsilon_{t}}$, a new approximation solution set of the piecewise  convexification optimization problem is presented.
\begin{theorem}\label{theorem:20210822}
For any $\varepsilon>0$, there exists $t\in \mathbb{N}$ and the subdivision $\mathbf{Y}^t$ of $X$ satisfied $$\max\limits_{k_t\in M_2(t)}\sum\limits_{i=1}^{n}{\alpha_{i}^{k_t}}\Big(\frac{a_{i}^{k_t}-b_{i}^{k_t}}{2}\Big)^2\leq\varepsilon$$ such that $\cX_{ap}(t)\subseteq \cX_{op}^{\varepsilon}$.
\end{theorem}
\begin{proof}
First, by the Lemma 5 in \cite{eichfelder2016modification}, there exists the subdivision $\mathbf{Y}^{t}$ such that $|T(\mathbf{Y}^{t})|\rightarrow 0$ as $t\rightarrow \infty$. Since $\alpha_{i}$ is a finite value for any $i$, then one can yield that there exist $t_0$ and a subdivision $\mathbf{Y}^{t_0}$ such that
\begin{align}\label{2021120801}
\max\limits_{k_{t_0}\in M_2(t_0)}\sum\limits_{i=1}^{n}{\alpha_{i}^{k_{t_0}}}\Big(\frac{a_{i}^{k_{t_0}}-b_{i}^{k_{t_0}}}{2}\Big)^2\leq\varepsilon.
\end{align}
Another step in the proof is $\cX_{ap}(t_0)\subseteq \cX_{op}^{\varepsilon}$. Suppose that  $\cX_{ap}(t_0)\nsubseteq\cX_{op}^{\varepsilon}$, that is, there exists $\hat{x}\in\cX_{ap}(t_0)$ such that $\hat{x}\notin\cX_{op}^{\varepsilon}$. Obviously, $\hat{x}\in\cX_{ap}^{C\cup NC}(t_0)$ and then one can find $\hat{y}\in X$ satisfied
\begin{align}\label{2022052201}
f(\hat{y})+\varepsilon< f(\hat{x}).
\end{align}
Without loss of generality, let $\hat{y}\in Y^{k_{t_0}}$  for $k_{t_0}\in M_{t_0}=M_{1}(t_0)\cup M_{2}(t_0)$.

If $\hat{y}\in\cX_{ap}^{k_{t_0}}\subseteq Y^{k_{t_0}}$, then according to (\ref{2022052201}) and $\cX_{ap}^{k_{t_0}}\subseteq\cX_{ap}^{C\cup NC}(t_0)$, it contradicts $\hat{x}\in\cX_{ap}(t_0)$. Thus $\hat{y}\in Y^{k_{t_0}}\backslash\cX_{ap}^{k_{t_0}}$.  Furthermore, there exists $\hat{z}\in \cX_{ap}^{k_{t_0}}$ satisfied  $F_{k_{t_0}}^{\alpha^{k_{t_0}}}(\hat{z})<F_{k_{t_0}}^{\alpha^{k_{t_0}}}(\hat{y})$ where  $k_{t_0}\in M_{t_0}$. Note that $M_{t_0}:=M_{1}(t_0)\cup M_{2}(t_0)$. Obviously, if $k_{t_0}\in M_{1}(t_0)$, i.e., $f$ is convex on $Y^{k_{t_0}}$, then $f(\hat{z})<f(\hat{y})$ and $f(\hat{z})< f(\hat{y})+\varepsilon <f(\hat{x})$ from (\ref{2022052201}).
This is contrary to $\hat{x}\in\cX_{ap}(t_0)$ as $\hat{z}\in\cX_{ap}^{k_{t_0}}\subset\cX_{ap}^{C}(t_0)$.
However, if $k_{t_0}\in M_{2}(t_0)$, then $\hat{y}\notin\cX_{ap}^{k_{t_0}}$ implies that there exists $\hat{z}\in\cX_{ap}^{k_{t_0}}\subset\cX_{ap}^{NC}(t_{0})$ such that $F_{k_{t_0}}^{\alpha^{k_{t_0}}}(\hat{z})< F_{k_{t_0}}^{\alpha^{k_{t_0}}}(\hat{y})$. Combining with (\ref{2022052201}), one can obtain  $F_{k_{t_0}}^{\alpha^{k_{t_0}}}(\hat{z})<f(\hat{x})-\varepsilon$, that is,
\begin{align*}
f(\hat{z})+\sum\limits_{i=1}^{n}{\alpha_{i}^{k_{t_0}}}(a_{i}^{k_{t_0}}-\hat{z}_{i}) (b_{i}^{k_{t_0}}-\hat{z}_{i})<f(\hat{x})-\varepsilon,
\end{align*}
which implies that $f(\hat{z})<f(\hat{x})$ by (\ref{2021120801}). This contradicts $\hat{x}\in\cX_{ap}(t_0)$ because of $\hat{z}\in\cX_{ap}^{k_{t_0}}\subset\cX_{ap}^{NC}(t_0)$. Consequently, we infer that $\hat{y}\notin Y^{k_{t_0}}\backslash\cX_{ap}^{k_{t_0}}$.

Apparently, the above proof are contrary to $\hat{y}\in Y^{k_{t_0}}$. Therefore, the theorem is now evident from what we have proved. 
\end{proof}

The above theorems  show that the solution set of the piecewise convexification optimization problem is a lower bound set of $\cX_{op}^{\varepsilon}$. In order to construct the upper bound set of $\cX_{op}^{\varepsilon}$, the approximation solution set of the piecewise convexification optimization problem is introduced, as defined by
\begin{align*}
\cX_{ap}^{\varepsilon}(t):=\left\{x\in\cX_{ap}^{C\cup NC,\varepsilon}(t): f(x)\leq f(y)+\varepsilon \text{~for~any~}y\in\cX_{ap}^{C\cup NC,\varepsilon}(t)\right\},
\end{align*}
where $\cX_{ap}^{C\cup NC,\varepsilon}(t):=\cX_{ap}^{C,\varepsilon}(t)\cup \cX_{ap}^{NC,\varepsilon}(t)$ and
\begin{align*}
\cX_{ap}^{C,\varepsilon}(t):=\bigcup\limits_{k_t\in M_{1}(t)}\cX_{ap}^{k_t,\varepsilon},~~
\cX_{ap}^{NC,\varepsilon}(t):=\bigcup\limits_{k_t\in M_{2}(t)}\cX_{ap}^{k_t,2\varepsilon},
\end{align*}
 $\cX_{ap}^{k_{t},\varepsilon}$ denotes the approximate optimal solution set of the convex relaxation sub-problem $F_{k_t}^{\alpha^{k_t}}$ on $Y^{k_t}$ with the quality $\varepsilon$, as defined by (\ref{cXk:20210725}).
 In what follows, we present that set $\cX_{ap}^{\varepsilon}(t)$ is a upper bound set of $\cX_{op}^{\varepsilon}$.
\begin{theorem}\label{theorem:2021082201}
For any $\varepsilon>0$, there exists $t\in \mathbb{N}$ and the subdivision $\mathbf{Y}^{t}$ of $X$ satisfied $$\max\limits_{k_t\in M_2(t)}\sum\limits_{i=1}^{n}\alpha_{i}^{k_t}\Big(\frac{b_{i}^{k_t}-a_{i}^{k_t}}{2}\Big)^2\leq\varepsilon$$ such that $\cX_{op}^{\varepsilon}\subseteq \cX_{ap}^{\varepsilon}(t)$.
\end{theorem}
\begin{proof}
Similarly,  there exist $t_0$ and a subdivision $\mathbf{Y}^{t_0}$ of $X$  such that (\ref{2021120801}) holds. It remains to show that $\cX_{op}^{\varepsilon}\subseteq \cX_{ap}^{\varepsilon}(t_0)$.  Assume that there exists $t_0$ such that $\cX_{op}^{\varepsilon}\nsubseteq\cX_{ap}^{\varepsilon}(t_{0})$. Then one can find $\hat{x}\in\cX_{op}^{\varepsilon}$ and  $\hat{x}\notin\cX_{ap}^{\varepsilon}(t_0)$. Now, $\hat{x}\notin\cX_{ap}^{\varepsilon}(t_0)$ can be distinguished two cases, the first of which is $\hat{x}\in\cX_{ap}^{C\cup NC,{\varepsilon}}(t_0)$. It implies that  there exists $\hat{y}\in\cX_{ap}^{C\cup NC,{\varepsilon}}(t_0)\subseteq X$ satisfied $f(\hat{y})+{\varepsilon}<f(\hat{x})$. This contradicts the fact $\hat{x}\in\cX_{op}^{\varepsilon}$, that is, the first case is not true.

The second case is $\hat{x}\notin\cX_{ap}^{C\cup NC,{\varepsilon}}(t_0)$. For this subdivision $\mathbf{Y}^{t_0}$, there exists $k_{t_0}\in M_{t_0}=M_1(t_0)\cup M_2(t_0)$ such that $\hat{x}\in Y^{k_{t_0}}$.  If $k_{t_0}\in M_{1}(t_0)$, then $\cX_{ap}^{k_{t_0},\varepsilon}\subseteq \cX_{ap}^{C,\varepsilon}(t_0)$. Moreover, $\hat{x}\notin\cX_{ap}^{C\cup NC,{\varepsilon}}(t_0)$ indicates that $\hat{x}\notin\cX_{ap}^{k_{t_0},\varepsilon}$. This means that there exists $\hat{y}\in Y^{k_{t_0}}$ satisfied $F_{k_{t_0}}^{\alpha_{k_{t_0}}}(\hat{y})+\varepsilon<F_{k_{t_0}}^{\alpha_{k_{t_0}}}(\hat{x})$. Obviously, $f$ is convex on $Y^{k_{t_0}}$ for $k_{t_0}\in M_{1}(t_0)$, that is, $F_{k_{t_0}}^{\alpha^{k_{t_0}}}(x)=f(x)$ for any $x\in Y^{k_{t_0}}$. Then  it conducts that $f(\hat{y})+\varepsilon=F_{k_{t_0}}^{\alpha^{k_{t_0}}}(\hat{y})+\varepsilon<F_{k_{t_0}}^{\alpha^{k_{t_0}}}(\hat{x})=f(\hat{x})$, which contradicts $\hat{x}\in\cX_{op}^{\varepsilon}$. This yields to $k_{t_0}\notin M_{1}(t_0)$, that is, $k_{t_0}\in M_{2}(t_0)$. This indicates that $\cX_{ap}^{k_{t_0},2\varepsilon}\subseteq \cX_{ap}^{NC,\varepsilon}(t_0)$ and $\hat{x}\notin\cX_{ap}^{k_{t_0},2\varepsilon}$ by $\hat{x}\notin\cX_{ap}^{C\cup NC,{\varepsilon}}(t_0)$. Thus there exists $\hat{y}\in Y^{k_t}$ satisfied $F_{k_{t_0}}^{\alpha^{k_{t_0}}}(\hat{y})+2\varepsilon<F_{k_{t_0}}^{\alpha^{k_{t_0}}}(\hat{x})$, which implies that $f(\hat{y})+\varepsilon<f(\hat{x})$ from (\ref{2021120801}).
This contradicts $\hat{x}\in\cX_{op}^{\varepsilon}$, which means that $k_{t_0}\in M_{2}(t_0)$ is also false. Consequently, the second case would not hold.

Therefore, the  assumption is not true, that is, $\cX_{op}^{\varepsilon}\subseteq\cX_{ap}^{\varepsilon}(t_{0})$ holds.
\end{proof}
\begin{remark}
Combing Theorems \ref{theorem:20210822} and \ref{theorem:2021082201}, the lower and upper bound sets of the approximate solution set of the original non-convex optimization problem are obtained by constructing solution sets of the piecewise convexification optimization problem, that is,
\begin{align*}
\cX_{ap}(t)\subseteq\cX_{op}^{\varepsilon}\subseteq \cX_{ap}^{\varepsilon}(t), \forall \varepsilon>0,
\end{align*}
where $t$ satisfies that $\max\limits_{k_t\in M_2(t)}\sum\limits_{i=1}^{n}{\alpha_{i}^{k_t}}\Big(\frac{a_{i}^{k_t}-b_{i}^{k_t}}{2}\Big)^2\leq\varepsilon$.
\end{remark}

{\section{The piecewise convexification algorithm}}
In this section, the piecewise convexification algorithm for the non-convex optimization problem is designed. Furthermore, we verity that this algorithm can output a subset of the approximate global optimal solutions set.
%
%
%
%
Some notations in the table must be introduced, which will be used in the algorithm.
\begin{table}[thp]\small
\begin{tabular}{ll}
\toprule[0.2mm]
Abbreviation & Denotation\\ \hline \specialrule{0em}{2pt}{2pt}
$\widehat{X}$   & The sub-box of $X$\\
$\hat{\mu}$  & The function values of $F_{\widehat{X}}^{\widehat{\alpha}}$ on $\widehat{X}$    \\
$\cX_{ap}^{\widehat{X}}$ & The optimal solution set of $F_{\widehat{X}}^{\widehat{\alpha}}$ on $\widehat{X}$\\
$v_{glob}$  & The smallest objective function value found for all current sub-boxes \\
$w(\widehat{X},\widehat{\alpha})$        & The modified width of $\widehat{X}$ and  $w(\widehat{X},\widehat{\alpha}):= \sum\limits_{i=1}^{n}{\widehat{\alpha}_{i}}\Big(\frac{\widehat{a}_{i}-\widehat{b}_{i}}{2}\Big)^2$ \\
$\tilde{\lambda}_{\widehat{X}}$  & A lower bound of $\lambda_{\min}(\widehat{X})$ computed by (\ref{lambda:2021051702})     \\
\specialrule{0em}{2pt}{2pt}
\toprule[0.2mm]
\end{tabular}
\end{table}

In what follows, we present the piecewise convexification algorithm to obtain a subset of the approximate global optimal solution set, which uses selection rule, discarding and termination tests,  as shown in Algorithm \ref{alg:Framwork}.
\begin{algorithm}[thb]\small
\caption{The piecewise convexification algorithm for {\rm(NCOP)}.}
\label{alg:Framwork}
\hspace*{0.02in} {\bf Input:} 
$X^0=[a,b]\in\mathbb{I}^{n}$, $f\in\mathbb{C}^2(\mathbb{R}^n,\mathbb{R}),\varepsilon>0;$
\hspace*{0.02in} {\bf Output:}
$\cX_{ap}^{new}$;
\begin{algorithmic}[1]
\State Compute an $\alpha^{0}$ of $f$ on $X^0$ according to (\ref{alpha:2021051201})-(\ref{sum:2021051701});
\label{ code:fram:extract }
\State Set $X^{*}:=X^0$, $x^{*}:=\frac{a+b}{2}$, $\mu^{*}:=-\infty$, $\alpha^{*}:=\alpha_0$, $x_{act}:=x^{*}$, $v_{glob}=v_{act}=+\infty$;
\State $L_{NC}:=\{(X^{*},x^{*},\mu^{*},\alpha^{*})\}$, $M_{C}=M_{D}=\emptyset$ and $k:=0$.
\While {$L_{NC}\neq\emptyset$ and $\max\limits_{(\widetilde{X},\widetilde{x},\widetilde{\mu},\widetilde{\alpha})\in L_{NC}}w(\widetilde{X},\widetilde{\alpha})>\varepsilon$}
\State $k:=k+1$;
\State Define $(X^{*},x^{*},\mu^{*},\alpha^{*})$ at the first element of $L_{NC}$ with $\max\limits_{(\widetilde{X},\widetilde{x},\widetilde{\mu},\widetilde{\alpha})\in L_{NC}}w(\widetilde{X},\widetilde{\alpha})$.
\State Delete $(X^{*},x^{*},\mu^{*},\alpha^{*})$ from $L_{NC}$.
\For{all $\widehat{X}\in Sp(X^{*})$}
\State Compute $\tilde{\lambda}_{\widehat{X}}$ by (\ref{lambda:2021051702}) and $\widehat{\alpha}$ by (\ref{alpha:2021051201}), respectively.
\State Compute
$\widehat{x}\in\arg\min\limits_{x\in\widehat{X}}F_{\widehat{X}}^{\widehat{\alpha}}(x)$. Let $\widehat{\mu}=\min\limits_{x\in\widehat{X}}F_{\widehat{X}}^{\widehat{\alpha}}(x)$.
\If {$\widehat{\mu}\leq v_{glob}$}
\If {$\tilde{\lambda}_{\widehat{X}}\geq 0$}
\State Add $(\widehat{X},\widehat{x},\widehat{\mu},\widehat{\alpha})$ as the last element to $M_{C}$ and add $\widehat{x}$ to $\cX_{ap}^{\widehat{X}}$.
\Else
\State Add $(\widehat{X},\widehat{x},\widehat{\mu},\widehat{\alpha})$ as the last element to $L_{NC}$ and add $\widehat{x}$ to $\cX_{ap}^{\widehat{X}}$.
\EndIf
\If {$f(\widehat{x})\leq v_{act}$}
\State Set $x_{act}=\widehat{x}, v_{act}=f(x_{act}),v_{glob}=\min\{v_{act},v_{glob}\}$
\State Delete $(X,x,\mu,\alpha)\in L_{NC}$ with $\mu>v_{glob}$ from $L_{NC}$, and
\State add $(X,x,\mu,\alpha)\in L_{NC}$ with $\mu>v_{glob}$ to $M_{D}$.
\EndIf
\Else
\State Add $(\widehat{X},\widehat{x},\widehat{\mu},\widehat{\alpha})$ as the last element to $M_{D}$.
\EndIf
\EndFor
\EndWhile
\State $M:=\bigcup\{\widetilde{X}:(\widetilde{X},\widetilde{x},\widetilde{\mu},\widetilde{\alpha})\in L_{UC}\}\cup M_{C}$. \vspace{0.2cm}
\State$\cX_{ap}^{new}:=\left\{x\in\bigcup\limits_{\widehat{X}\in M}\cX_{ap}^{\widehat{X}}:~ f(x)\leq f(y), ~~\forall ~ y\in\bigcup\limits_{\widehat{X}\in M}\cX_{ap}^{\widehat{X}}\right\}$.
\end{algorithmic}
\end{algorithm}

This algorithm terminates after finitely many iterations because there exists the number of subdivisions $t$ such that $\max\limits_{(\widetilde{X},\tilde{x},\tilde{\mu},\tilde{\alpha})\in L_{NC}}w(\widetilde{X},\tilde{\alpha})\leq\varepsilon$.
Noting that this algorithm applies the same discarding technique to deleting sub-boxes as the modified $\alpha$BB method \cite{eichfelder2016modification}. Thus it holds that $\widehat{X}\cap \cX_{op}=\emptyset$ when $\hat{\mu}>v_{glob}$. However, the termination conditions and the selection way of sub-box for the division are different. Compared with \cite{eichfelder2016modification}, there is only one termination condition, that is, $L=\emptyset$. But this piecewise convexification algorithm sets two termination conditions,  that is,
 $\max\limits_{(\widetilde{X},\tilde{x},\tilde{\mu},\tilde{\alpha})\in L_{NC}}w(\widetilde{X},\widetilde{\alpha})\leq\varepsilon$ or $L_{NC}=\emptyset$. Clearly, if $L_{NC}=\emptyset$, then $\max\limits_{(\widetilde{X},\tilde{x},\tilde{\mu},\tilde{\alpha})\in L_{NC}}w(\widetilde{X},\tilde{\alpha})=0$, and not vice versa, i.e., $L_{NC}\neq \emptyset$ may be hold  when  $\max\limits_{(\widetilde{X},\tilde{x},\tilde{\mu},\tilde{\alpha})\in L_{NC}}w(\widetilde{X},\tilde{\alpha})\leq\varepsilon$. 
The following numerical experiment results, in section 5, will show that for some complex problems, these two termination conditions are more conducive to speeding up the algorithm than having only one termination condition in the modified $\alpha$BB method \cite{eichfelder2016modification}. In this algorithm, we put a new selection rule of sub-box for the division, that is, the box with the maximum  modified width in $L_{NC}$  is selected to divide into two sub-boxes as shown in line 5. This selection approach is different from the one proposed in \cite{eichfelder2016modification}. Moreover, different from the modified $\alpha$BB method \cite{eichfelder2016modification}, this piecewise convexification algorithm only requires a while loop and does not require additional parameters.

\begin{remark}\label{remark:20210922}
From the piecewise convexification algorithm framework,  the union of sets $\bigcup\limits_{(\widehat{X},\widehat{x},\widehat{\mu},\widehat{\alpha})\in M_{1}}\widehat{X}$ is a subdivision of $X$ where $M_{1} = M \cup L_{NC}$.
\end{remark}
The following theorem is illustrated that the output of this algorithm $\cX_{ap}^{new}$ is a subset of the approximate global solution set.
\begin{theorem}
At the end of the algorithm 1 the set $\cX_{ap}^{new}$ is a subset of the approximate global optimal solution set $\cX_{op}^{\varepsilon}$.
\end{theorem}
\begin{proof}
The proof process is similar to Theorem \ref{theorem:20210822} and will be omitted.
\end{proof}

\section{Numerical Experiments}
In this section, we  demonstrate the efficiency of this piecewise convexification algorithm and compare it with $\text{the~mod}~\alpha_{i,d=u-l}^{loc} BB$\upcite{eichfelder2016modification}. All computations have been performed on a computer with Iter(R)Core(TM)i5-8250U CPU and 8 Gbytes RAM. In numerical results, we use the following notations.

\begin{table}[thp]
\begin{tabular}{ll}
\toprule[0.2mm]
Abbreviation & Denotation                                              \\ \hline
$iter$        & Number of required iterations                           \\
CPU & Required CPU time in seconds\\
$N_{\varepsilon}$  & Number of $\varepsilon$-optimal solutions from algorithm, let $\varepsilon=10^{-3}$ \\
$flag_{ter}$ & The indicator of termination condition of algorithm, $flag_{ter}=1\text{~or~}0$ \\
PC-NCOP & The piecewise convexification algorithm for (NCOP), i.e. Algorithm \ref{alg:Framwork}\\
-            & The algorithm does not record a certain value \\
\toprule[0.2mm]
\end{tabular}
\end{table}

In this paper, $flag_{ter}=1$ indicates that the termination condition of this algorithm is $L_{NC}=\emptyset$, otherwise, $flag_{ter}=0$ when $\max\limits_{(\tilde{X},\tilde{x},\tilde{\mu},\tilde{\alpha})\in L_{NC}}w(\tilde{X},\tilde{\alpha})\leq\varepsilon.$ Moreover, we replace $\bar{\mu}\leq v_{glob}$ in line 11 by $\bar{\mu}\leq v_{glob}+10^{-6}$ for all instances.

As stated in \cite{eichfelder2016modification}, authors use INTLAB ToolBox to automaticly compute the elements $\nabla^2f(X)_{ij}$.  In fact, we directly solve the optimization problem $\min\limits_{x\in\bar{X}}\nabla^2f(\bar{X})_{ij}$, to estimate $\alpha_{i}^{k_n}$ on each sub-box rather than using INTLAB.
Therefore, we rewrite the modified {\rm$\alpha$BB} method \cite{eichfelder2016modification}, without INTLAB.
First, we demonstrate the performance of two approaches on nine test instances with finite optimal solutions, in which most test instances have multiple optimal solutions and all examples are two-dimensional. Some information of these test instances about the objective functions $f$, feasible sets $X^{0}$, number of globally solutions, and global optimal values are listed in Tables \ref{table:f1}.
\begin{table}[thp]
\centering
\small
\setlength{\abovecaptionskip}{0.1cm}
\setlength{\belowcaptionskip}{0.1cm}
\captionsetup{font={small}}
\caption{Test instances with finite number of optimal solutions}
\label{table:f1}
\centering 
\resizebox{\textwidth}{!}
{
\begin{tabular}{lllcc} 
\toprule[0.3mm]
\specialrule{0em}{2pt}{2pt}
&\multicolumn{1}{c} {$f: \mathbb{R}^{2}\rightarrow \mathbb{R}$ with $f(x)=$} & \multicolumn{1}{c} {$X$} & $|\arg\min\limits_{x\in X}f(x)|$ & \multicolumn{1}{c}{$\min\limits_{x\in X}f(x)$}
\\\specialrule{0em}{2pt}{2pt}
\hline 
\specialrule{0em}{3pt}{3pt}
{\bf Rastrigin}\upcite{eichfelder2016modification} &$20+x_{1}^{2}+x_{2}^{2}-10(\cos(2\pi x_{1})+\cos(2\pi x_{2}))$ & $\left[\left(
 \begin{matrix}-5.12\\-5.12\end{matrix}\right),\left(\begin{matrix}5.12\\5.12\end{matrix}\right)\right]$
& 1 & 0 \\
\specialrule{0em}{2pt}{2pt}
{\bf 6-Hump}\upcite{epitropakis2011finding} &$(4-2.1x_{1}^{2}+\frac{1}{3}x_{1}^4)x_{1}^{2}+x_{1}x_{2}-(4-4x_{2}^2)x_{2}^{2}$ & $\left[\left(
 \begin{matrix}-1.9\\-1.1\end{matrix}\right),\left(\begin{matrix}1.9\\1.1\end{matrix}\right)\right]$
& 2 & $\approx-1.031629$ \\
\specialrule{0em}{2pt}{2pt}
{\bf Branin}\upcite{epitropakis2011finding} &$(x_{2}-\frac{5.1}{4\pi^2}x_{1}^{2}+\frac{5}{\pi}x_{1}-6)^2+10(1-\frac{1}{8\pi})\cos(x_{1})+10$ & $\left[\left(
 \begin{matrix}-5\\0\end{matrix}\right),\left(\begin{matrix}10\\15\end{matrix}\right)\right]$
& 3 & $\approx 0.397886$ \\
\specialrule{0em}{2pt}{2pt}
{\bf Himmelblau}\upcite{epitropakis2011finding} &$(x_{1}^{2}+x_{2}-11)^2+(x_{1}+x_{2}^{2}-7)^2$ & $\left[\left(
 \begin{matrix}-6\\-6\end{matrix}\right),\left(\begin{matrix}6\\6\end{matrix}\right)\right]$
& 4 & 0 \\
\specialrule{0em}{2pt}{2pt}
{\bf Rastrigin mod}\upcite{epitropakis2011finding} &$20+x_{1}^{2}+x_{2}^{2}+10(\cos(2\pi x_{1})+\cos(2\pi x_{2}))$ & $\left[\left(
 \begin{matrix}-5.12\\-5.12\end{matrix}\right),\left(\begin{matrix}5.12\\5.12\end{matrix}\right)\right]$
& 4 & $\approx0.497480$ \\
\specialrule{0em}{2pt}{2pt}
{\bf Shubert}\upcite{epitropakis2011finding} &$\sum\limits_{i=1}^{5}[i\cos((i+1)x_{1}+1)]\sum\limits_{j=1}^{5}[j\cos((j+1)x_{2}+j)]$ & $\left[\left(
 \begin{matrix}-10\\-10\end{matrix}\right),\left(\begin{matrix}10\\10\end{matrix}\right)\right]$
& 18 & $\approx-186.730909$ \\
\specialrule{0em}{2pt}{2pt}
{\bf Deb 1} \upcite{epitropakis2011finding} &$-\frac{1}{2}(\sin^{6}(5\pi x_{1})+\sin^{6}(5\pi x_{2})) $ & $\left[\left(
 \begin{matrix}0\\0\end{matrix}\right),\left(\begin{matrix}1\\1\end{matrix}\right)\right]$
& 25 & $-1$ \\
\specialrule{0em}{2pt}{2pt}
{\bf Vincent}\upcite{epitropakis2011finding} &$-\frac{1}{2}(\sin(10\ln(x_{1}))+\sin(10\ln(x_{2})))$ & $\left[\left(
 \begin{matrix}0.25\\0.25\end{matrix}\right),\left(\begin{matrix}10\\10\end{matrix}\right)\right]$
& 36 & $-1$ \\
\specialrule{0em}{2pt}{2pt}
\hline
\toprule[0.2mm]
\end{tabular}
}
\end{table}

Numerical results of these two algorithms are presented in Table \ref{Numerical1}.
\begin{table}[thp]
\setlength{\abovecaptionskip}{0.1cm}
\setlength{\belowcaptionskip}{0.1cm}
\captionsetup{font={small}}
\caption{Numerical results for test instances in Table \ref{table:f1}} 
\centering\
\resizebox{\textwidth}{!}
{
\label{Numerical1}
\setlength{\tabcolsep}{5mm}{
\begin{tabular}{lllll}
\toprule[0.2mm]
\multicolumn{5}{c}{PC-NCOP/$\text{the~mod}~\alpha_{i,d=u-l}^{loc} BB$\upcite{eichfelder2016modification}}\\ \hline
     &\multicolumn{1}{c} {$iter$} &\multicolumn{1}{c} {CPU} &\multicolumn{1}{c}{$N_{\varepsilon}$} &\multicolumn{1}{c} {$flag_{ter}$} \\ \hline
{\bf Rastrigin}     &{\bf 104}/551   &{\bf 1.045}/30.959    &{\bf 1}/{\bf 1}      & 1/1 \\\specialrule{0em}{1pt}{1pt}
{\bf 6-Hump}        &{\bf47}/49     &{\bf0.402}/0.418        &{\bf2}/{\bf1}      & 1/1\\\specialrule{0em}{1pt}{1pt}
{\bf Branin}        &{\bf52}/67     &{\bf0.576}/0.726       &2/2     & 0/1\\\specialrule{0em}{1pt}{1pt}
{\bf Himmelblau}    &{\bf43}/382     &{\bf0.558}/4.345       &{\bf4}/{\bf4}      & 0/1 \\\specialrule{0em}{1pt}{1pt}
{\bf Rastrigin mod} &{\bf571}/859   &{\bf6.734}/11.046   &{\bf4}/{\bf 4}          & 1/1 \\\specialrule{0em}{1pt}{1pt}
{\bf Shubert}       &{\bf3091}/5056 &{\bf56.073}/136.460    &{\bf18}/{\bf18}    & 0/1 \\\specialrule{0em}{1pt}{1pt}
{\bf Deb 1}         &{\bf391}/863   &{\bf5.390}/17.009    &{\bf25}/{\bf25}    & 1/1 \\\specialrule{0em}{1pt}{1pt}
{\bf Vincent}       &{\bf1169}/11705 &{\bf10.820}/166.680    &{\bf36}/{\bf 36}   & 1/1 \\\specialrule{0em}{1pt}{1pt}
\toprule[0.2mm]
\end{tabular}
}}
\end{table}
It is easy to see that for all the test examples, the  $iter$ and CPU values of PC-NCOP are significantly less than those of $\text{the~mod}~\alpha_{i,d=u-l}^{loc} BB$\upcite{eichfelder2016modification}. As for the $|N_{\varepsilon}|$ values, only a slight difference for {\bf Branin} exists, that is, two algorithms can only find two globally optimal solutions, not three. These results demonstrate that the proposed algorithm can find almost all the optimal solutions of the  original non-convex problem, and PC-NCOP is better than $\text{the~mod}~\alpha_{i,d=u-l}^{loc} BB$\upcite{eichfelder2016modification}. In addition, for {\bf Branin, Himmelblau} and {\bf Shubert} the termination condition of PC-NCOP is the same, that is, $\max\limits_{(\tilde{X},\tilde{x},\tilde{\mu},\tilde{\alpha})\in L_{NC}}w(\tilde{X},\tilde{\alpha})\leq\varepsilon$. The termination condition of PC-NCOP for other remaining instances is $\ L_{NC}=\emptyset$.

The division results of $X$ are clearly shown in Fig. \ref{figure: finite test instances},
\begin{figure}[htbp]
\setlength{\abovecaptionskip}{0.1cm}
\setlength{\belowcaptionskip}{0.1cm}
\centering
\captionsetup{font={scriptsize}}
\subfigure[Rastrigin]{
\includegraphics[scale=0.1425]{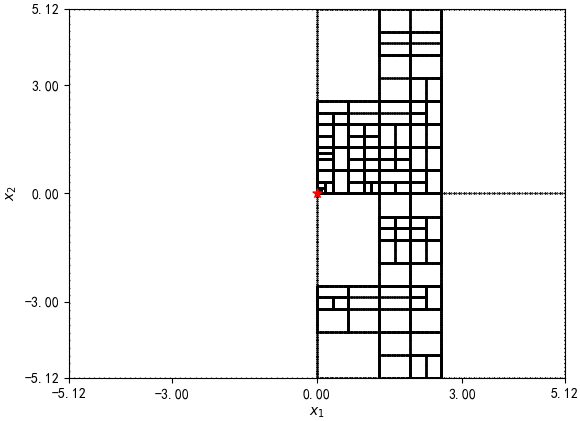} \label{1}
}\hspace{-3mm}
\subfigure[Rastrigin\upcite{eichfelder2016modification}]{
\includegraphics[scale=0.1425]{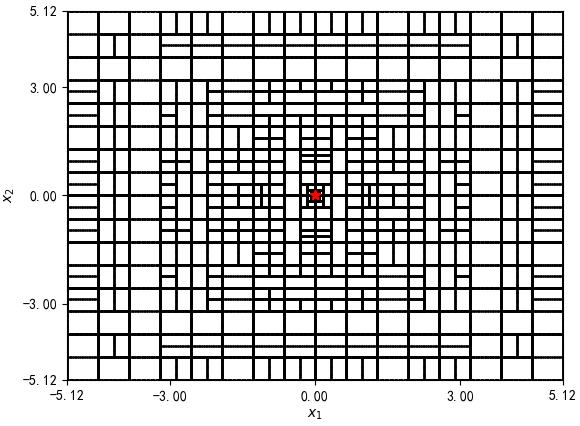}
}\hspace{-3mm}
\subfigure[6-Hump]{
\includegraphics[scale=0.1425]{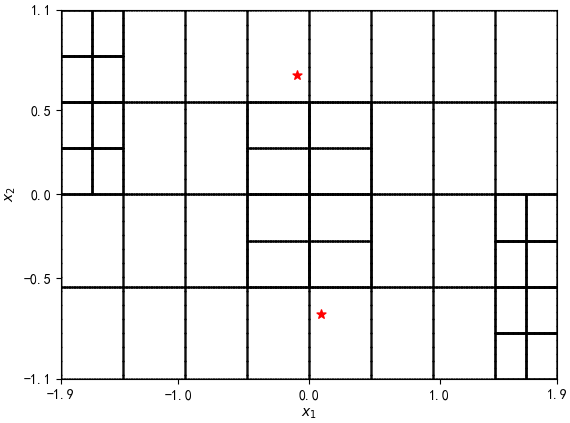} \label{2}
}
\hspace{-3mm}
\subfigure[6-Hump\upcite{eichfelder2016modification}]{
\includegraphics[scale=0.1425]{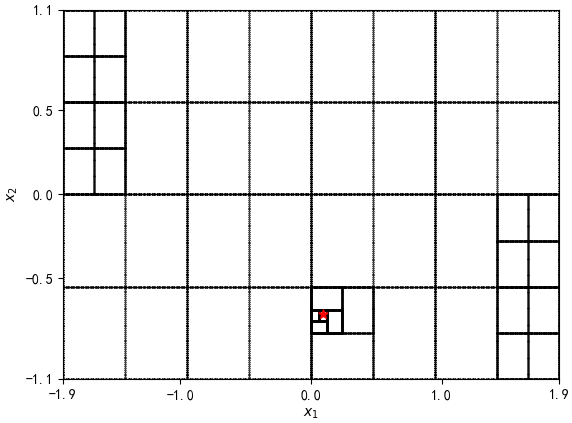} \label{2}
}
\hspace{-3mm}
\subfigure[Brain]{
\includegraphics[scale=0.144]{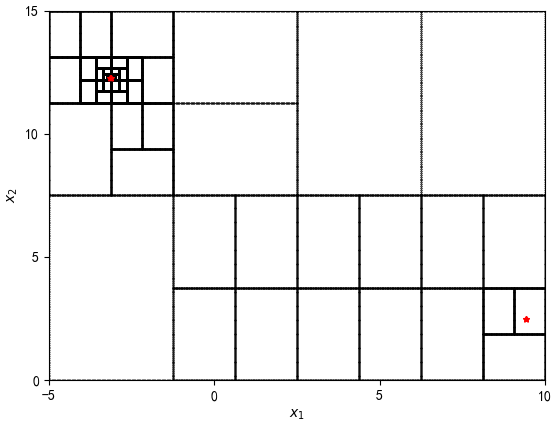}\label{3}
}
\hspace{-3mm}
\subfigure[Brain\upcite{eichfelder2016modification}]{
\includegraphics[scale=0.144]{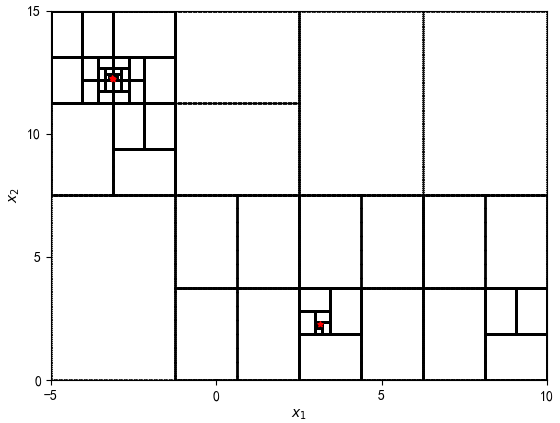}\label{3}
}
\hspace{-3mm}
\subfigure[Himmelblau]{
\includegraphics[scale=0.144]{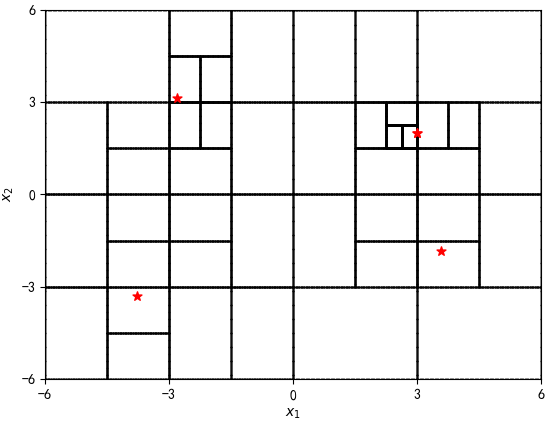}\label{4}
}
\hspace{-3mm}
\subfigure[Himmelblau\upcite{eichfelder2016modification}]{
\includegraphics[scale=0.144]{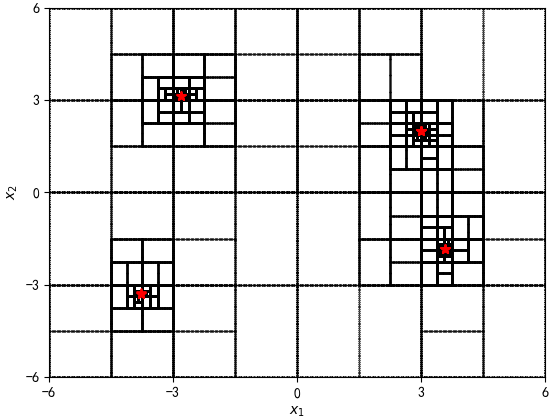}\label{4}
}
\hspace{-4mm}
\subfigure[Rastrigin mod]{
\includegraphics[scale=0.141]{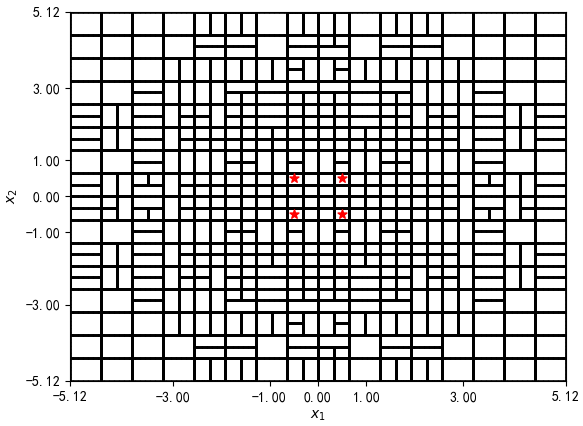}\label{4}
}
\hspace{-4mm}
\subfigure[Rastrigin mod\upcite{eichfelder2016modification}]{
\includegraphics[scale=0.141]{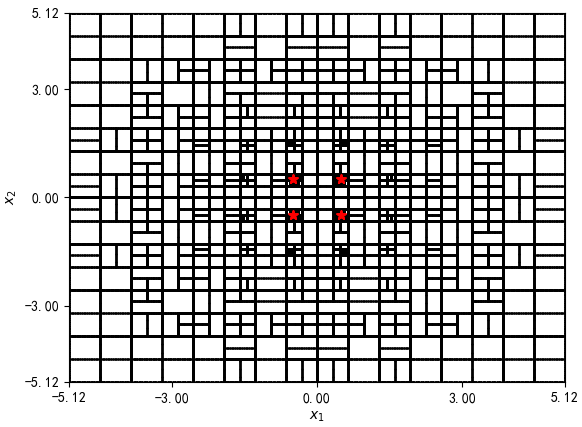}\label{4}
}
\hspace{-4mm}
\subfigure[Shubert]{
\includegraphics[scale=0.141]{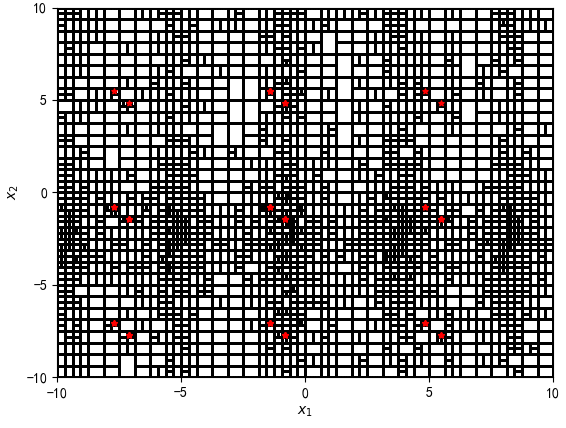}\label{4}
}
\hspace{-4mm}
\subfigure[Shubert\upcite{eichfelder2016modification}]{
\includegraphics[scale=0.141]{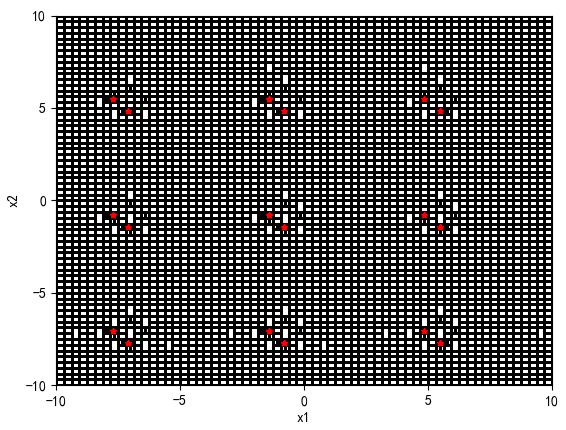}\label{4}
}
\hspace{-4mm}
\subfigure[Deb 1]{
\includegraphics[scale=0.142]{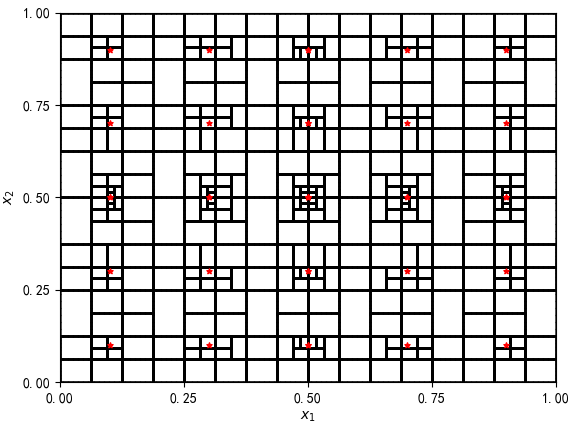}\label{4}
}
\hspace{-4mm}
\subfigure[Deb 1\upcite{eichfelder2016modification}]{
\includegraphics[scale=0.142]{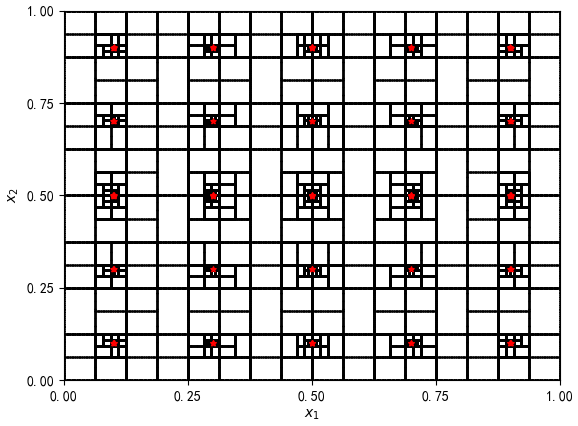}\label{4}
}
\hspace{-4mm}
\subfigure[Vincent]{
\includegraphics[scale=0.142]{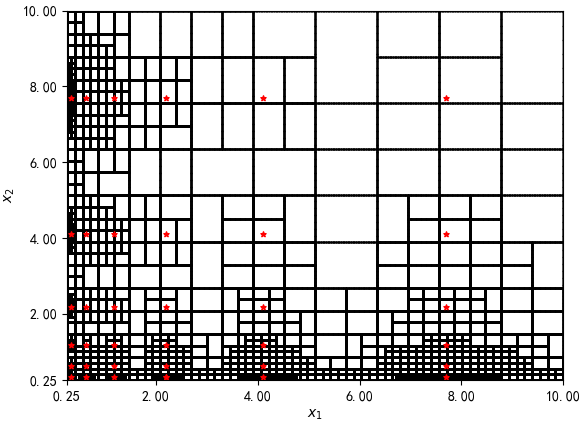}\label{4}
}
\hspace{-4mm}
\subfigure[Vincent\upcite{eichfelder2016modification}]{
\includegraphics[scale=0.142]{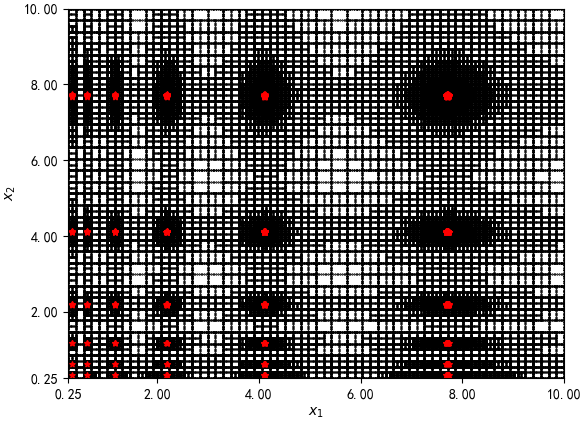}\label{4}
}
\captionsetup{font={small}}
\caption{Interval subdivision results of test instances in Table \ref{table:f1}}
\label{figure: finite test instances}
\end{figure}
\vspace{-0.2em}
where the first and third columns show the results  obtained by Algorithm \ref{alg:Framwork}, while the results of $\text{the~mod}~\alpha_{i,d=u-l}^{loc} BB$\upcite{eichfelder2016modification} are shown in the second and fourth columns. The red star denotes the $\varepsilon$-optimal solutions. Obviously,  it can be seen from Fig. \ref{figure: finite test instances} that the selection way of the boxes to be divided can effectively reduce the number of iterations. In fact, $\text{the~mod}~\alpha_{i,d=u-l}^{loc} BB$\upcite{eichfelder2016modification} has numerous subdivisions of the box near the optimal solution, while our algorithm has only a few subdivisions, and the following partial graph intuitively reflects this assertion.
\begin{figure}[thp]
\centering
\subfigure[Rastrigin mod]{
\includegraphics[width=5.75cm]{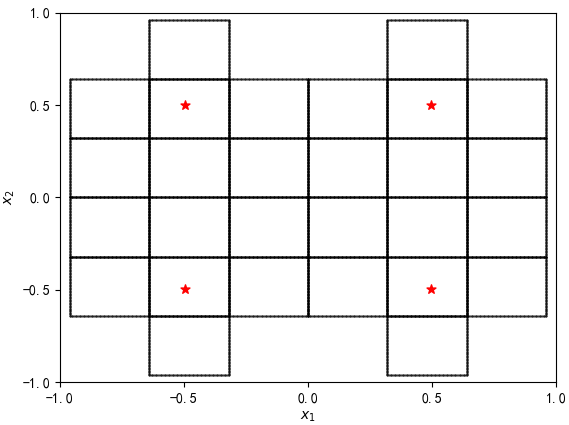}
}\hspace{-5mm}
\quad
\subfigure[Rastrigin mod \cite{eichfelder2016modification}]{
\includegraphics[width=5.75cm]{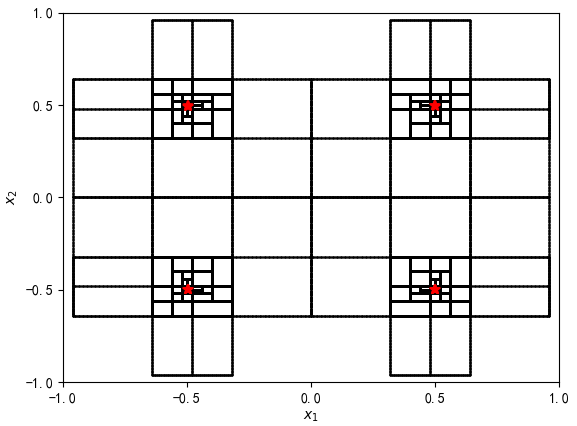}
}
\quad
\subfigure[Deb 1]{
\includegraphics[width=5.75cm]{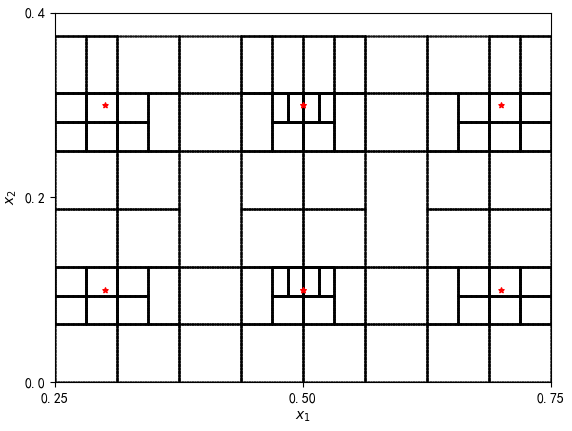}
}\hspace{-5mm}
\quad
\subfigure[Deb 1\cite{eichfelder2016modification}]{
\includegraphics[width=5.75cm]{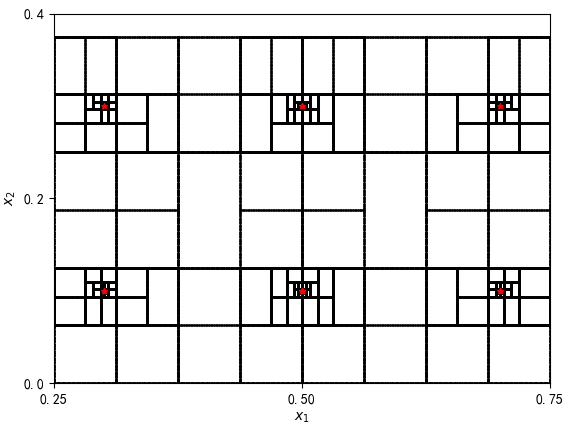}
}
\captionsetup{font={small}}
\caption{Subdivisions on subinterval for Algorithm \ref{alg:Framwork}}
\label{fig:subdivision}
\end{figure}

In what follows, we consider four numerical tests with infinite number of globally optimal solutions listed in \cite{eichfelder2016modification}, as defined by Table \ref{tab:2}.
\begin{table}[thp]
\setlength{\abovecaptionskip}{0.05cm}
\setlength{\belowcaptionskip}{0.05cm}
\captionsetup{font={small}}
\caption{Test instances with infinite number of optimal solutions} 
\centering 
\resizebox{\textwidth}{!}
{
\begin{tabular}{llcl} 
\toprule[0.2mm]
\specialrule{0em}{1pt}{1pt}
 &{$f:\mathbb{R}^{2}\rightarrow\mathbb{R}$} &{ $X$} &{$\arg\min\limits_{x\in X}f(x)$}\\\specialrule{0em}{1pt}{1pt}
\hline
\specialrule{0em}{2pt}{2pt}
{\bf Test01} & $\left(\frac{x_{1}^{2}}{4^2}+\frac{x_{2}^{2}}{2^2}-1\right)^2$ & $\left[\left(
 \begin{matrix}-5\\-5\end{matrix}\right),\left(\begin{matrix}5\\5\end{matrix}\right)\right]$ & $\left\{\left(
 \begin{matrix}x_{1}\\x_{2}\end{matrix}\right)\Big|\frac{x_{1}^{2}}{4^2}+\frac{x_{2}^{2}}{2^2}=1\right\}$\\
\specialrule{0em}{1pt}{1pt}
{\bf Test02}& $\frac{1}{10}(x_{1}(1-x_{2})+x_{2}(1-x_{1}))^2$ & $\left[\left(
 \begin{matrix}-5\\-5\end{matrix}\right),\left(\begin{matrix}5\\5\end{matrix}\right)\right]$ & $\left\{\left(
 \begin{matrix}x_{1}\\x_{2}\end{matrix}\right)\Big|x_{1}\in[-5, 5]\backslash
 \{\frac{1}{2}\}, x_{2}=-\frac{x_{1}}{1-2x_{1}}\right\}$\\
\specialrule{0em}{1pt}{1pt}
{\bf Test03}& $\sin^{2}(\frac{5}{4}x_{1}+x_{2}-3)$ & $\left[\left(
 \begin{matrix}0\\-2\end{matrix}\right),\left(\begin{matrix}4\\3\end{matrix}\right)\right]$ & $\left\{\left(
 \begin{matrix}x_{1}\\x_{2}\end{matrix}\right)\Big|\frac{5}{4}x_{1}+x_{2}=3+a, a\in\{0,\pm\pi\}\right\}\bigcap X$\\
\specialrule{0em}{1pt}{1pt}
{\bf Test04}& $(x_{1}+\sin^{2}(x_{1}))\cos^{2}(x_{2})$ & $\left[\left(
 \begin{matrix}0\\-2\end{matrix}\right),\left(\begin{matrix}4\\3\end{matrix}\right)\right]$ & $\left\{\left(
 \begin{matrix}0\\x_{2}\end{matrix}\right)\Big|x_{2}\in[-2,3]\right\}\bigcup \left\{\left(
 \begin{matrix}x_{1}\\x_{2}\end{matrix}\right)\Big| x_{1}\in[0,4],x_{2}\in\{\pm\frac{\pi}{2}\}\right\}$\\
\specialrule{0em}{2pt}{2pt}
\toprule[0.2mm]
\end{tabular}
}
\label{tab:2}
\end{table}

Table \ref{table:2021051801} shows the numerical results of the instance tests of Table \ref{tab:2}.
\begin{table}[thp]\small
\setlength{\abovecaptionskip}{0.05cm}
\setlength{\belowcaptionskip}{0.05cm}
\captionsetup{font={small}}
\caption{Numerical results for test instance in Table \ref{tab:2}} 
\centering
\setlength{\tabcolsep}{5.4mm}{
\begin{tabular}{lllll}
\toprule[0.15mm]
\multicolumn{5}{c}{PC-NCOP/$\text{the~mod}~\alpha_{i,d=u-l}^{loc} BB$\upcite{eichfelder2016modification}}\\ \toprule[0.1mm]
& \multicolumn{1}{c}{$iter$} &\multicolumn{1}{c}{CPU} &\multicolumn{1}{c}{$N_{\varepsilon}$} &\multicolumn{1}{c}{$flag_{ter}$} \\ \toprule[0.1mm]
{\bf Test01}& {\bf559}/1355 &{\bf6.717}/18.41  &{\bf592}/588  & 0/1 \\\specialrule{0em}{1pt}{1pt}
{\bf Test02}& {\bf672}/1156 &{\bf6.891}/17.511  &{\bf649}/433  & 0/1 \\\specialrule{0em}{1pt}{1pt}
{\bf Test03}& {\bf1189}/3019  &{\bf11.511}/52.353  &1237/{\bf1336}  & 0/1 \\\specialrule{0em}{1pt}{1pt}
{\bf Test04}& {\bf2343}/4863  &{\bf22.724}/121.239  &{\bf3226}/2123  & 0/1 \\\specialrule{0em}{1pt}{1pt}
\toprule[0.15mm]
\end{tabular}}
\label{table:2021051801}
\end{table}
The $iter$ and CPU values of PC-NCOP are significantly better than the value of $\text{the~mod}~\alpha_{i,d=u-l}^{loc} BB$\upcite{eichfelder2016modification}. Except for {\bf Test03}, the number of solutions of PC-NCOP is also higher than that of $\text{the~mod}~\alpha_{i,d=u-l}^{loc} BB$\upcite{eichfelder2016modification}. In addition, the termination condition $\max\limits_{(\widehat{X},\hat{x},\hat{\mu},\hat{\alpha})}w(\widehat{X},\hat{\alpha})\leq\varepsilon$, in PC-NCOP, is satisfied for these test problems  as $flag_{ter}=0$. In other words, this termination condition is meaningful in Algorithm \ref{alg:Framwork} and could be helpful to reduce the number of iterations. Moreover, the results of the interval subdivision and solutions in set $\cX_{ap}^{new}$ for Algorithm \ref{alg:Framwork} are showed in Fig.\ref{fig:subdivision}.
\begin{figure}[thp]
\centering
\setlength{\abovecaptionskip}{0.05cm}
\setlength{\belowcaptionskip}{0.05cm}
\subfigure[{\bf Test01}]{
\includegraphics[width=5.75cm]{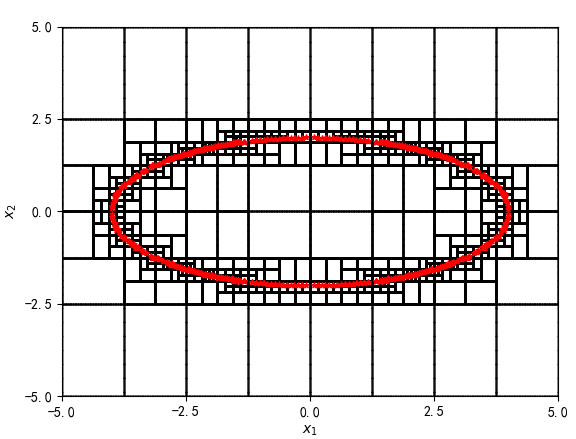}
}\hspace{-5mm}
\quad
\subfigure[{\bf Test02}]{
\includegraphics[width=5.75cm]{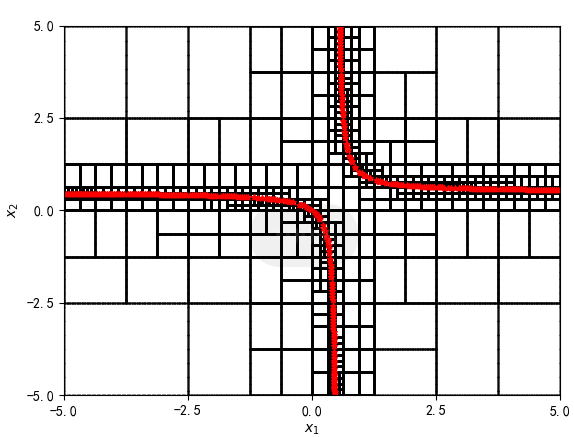}
}
\quad
\subfigure[{\bf Test03}]{
\includegraphics[width=5.75cm]{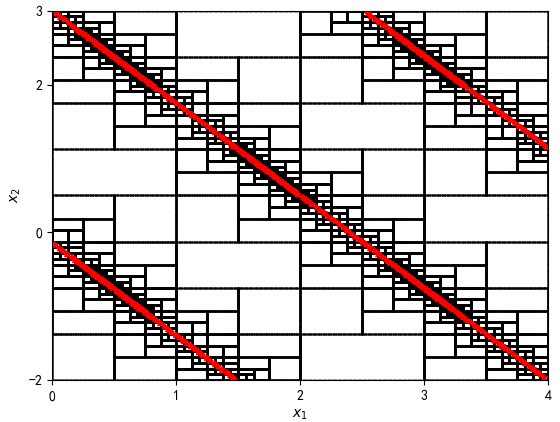}
}\hspace{-5mm}
\quad
\subfigure[{\bf Test04}]{
\includegraphics[width=5.75cm]{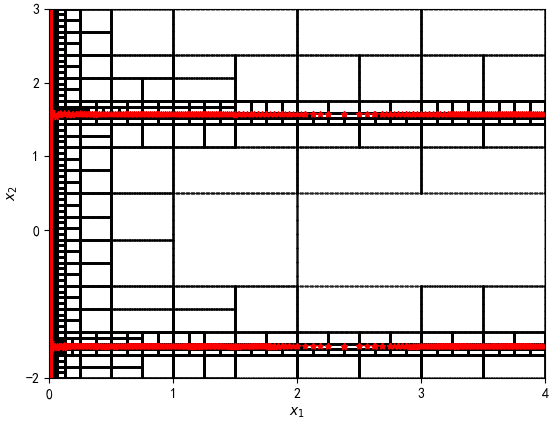}
}
\captionsetup{font={small}}
\caption{Sub-boxes and solutions in set $\cX_{ap}^{new}$ for Algorithm \ref{alg:Framwork} and {\bf Test01-04}}
\label{fig:subdivision}
\end{figure}
This figure  shows that the distribution of these optimal solutions obtained from the PC-NCOP can be used to describe the distribution of the optimal solutions of the original problem.

Finally, in order to verify the efficiency of the proposed algorithm for the high-dimensional instances, Table \ref{tab:test instance with high dimension} is shown a high-dimensional test problem, which is selected from the literature \cite{eichfelder2016modification}.
\begin{table}[thp]
\setlength{\abovecaptionskip}{0.1cm}
\setlength{\belowcaptionskip}{0.1cm}
\captionsetup{font={small}}
\caption{Test instances with high dimensional} 
\centering 
\resizebox{\textwidth}{!}
{
\begin{tabular}{llcc} 
\toprule[0.2mm]
 & $f:\mathbb{R}^{2}\rightarrow\mathbb{R}$ & $X$ & $\arg\min\limits_{x\in X}f(x)$\\
\hline
\specialrule{0em}{2pt}{2pt}
{\bf TestDim$_{d}$},$d\in\mathbb{N}$\upcite{eichfelder2016modification} &$\sum\limits_{i=1}^{d}(\cos(2\pi x_{i}))^2$& $\left[\left(
 \begin{matrix}-\frac{1}{4}\\ \cdots\\-\frac{1}{4}\end{matrix}\right),\left(\begin{matrix}\frac{1}{4}\\ \cdots\\\frac{1}{4}\end{matrix}\right)\right]$ & $\left\{x\in\mathbb{R}^{d}\Big| x_{i}\in\{-\frac{1}{4},\frac{1}{4}\},i\in\{1,\cdots,d\}\right\}$\\
 \specialrule{0em}{2pt}{2pt}
\toprule[0.2mm]
\end{tabular}
}
\label{tab:test instance with high dimension}
\end{table}

Table \ref{table: results for high dimensional} gives results by applying the PC-NCOP and $\text{the~mod}~\alpha_{i,d=u-l}^{loc} BB$\upcite{eichfelder2016modification}  to {\bf TestDim$_{d}$}.
\begin{table}[thp]
\setlength{\abovecaptionskip}{0.1cm}
\setlength{\belowcaptionskip}{0.1cm}
\captionsetup{font={small}}
\caption{Numerical results for test instance in Table \ref{tab:test instance with high dimension}} 
\centering
\setlength{\tabcolsep}{4.4mm}{
\begin{tabular}{lllll}
\toprule[0.2mm]
\multicolumn{5}{c}{PC-NCOP/$\text{the~mod}~\alpha_{i,d=u-l}^{loc} BB$\upcite{eichfelder2016modification}}\\ \hline
     & $iter$ & CPU & $N_{\varepsilon}$  & $flag_{ter}$ \\ \hline
{\bf TestDim}$_{d=2}$      &{\bf11}/47        &{\bf0.061}/0.316          &4/4      &0/1 \\\specialrule{0em}{2pt}{2pt}
{\bf TestDim}$_{d=3}$      &{\bf47}/192        &{\bf0.491}/2.307        &8/8      &0/1 \\\specialrule{0em}{2pt}{2pt}
{\bf TestDim}$_{d=4}$      &{\bf175}/655      &{\bf2.915}/12.994       &16/16    &0/1 \\\specialrule{0em}{2pt}{2pt}
{\bf TestDim}$_{d=5}$      &{\bf607}/2076      &{\bf17.1075}/84.313      &32/32    &0/1 \\\specialrule{0em}{2pt}{2pt}
{\bf TestDim}$_{d=6}$      &{\bf2047}/8007    &{\bf78.059}/267.484   &64/64    &0/1 \\
\specialrule{0em}{2pt}{2pt}
{\bf TestDim}$_{d=7}$      &{\bf6783}/28353    &{\bf420.671}/1223.365   &128/128  &0/1 \\\specialrule{0em}{2pt}{2pt}
{\bf TestDim}$_{d=8}$      &{\bf22272}/89871  &{\bf1103.478}/5391.845  &256/256  &0/1 \\\specialrule{0em}{2pt}{2pt}
{\bf TestDim}$_{d=9}$      &{\bf72704}/282072  &{\bf4940.783}/32916.650   &512/512  &0/1 \\
\toprule[0.2mm]
\end{tabular}}
\label{table: results for high dimensional}
\end{table}
From Table \ref{table: results for high dimensional}, the experimental results demonstrate that both the proposed algorithm and $\text{the~mod}~\alpha_{i,d=u-l}^{loc} BB$\upcite{eichfelder2016modification} lead to the same values of $N_{\varepsilon}$. However, compared with Algorithm \ref{alg:Framwork}, $\text{the~mod}~\alpha_{i,d=u-l}^{loc} BB$\upcite{eichfelder2016modification} requires more iterations and CPU, and the advantage of Algorithm \ref{alg:Framwork} is more prominent as the dimension increases. In these high-dimensional problems, $flag_{ter}=0$ indicates that the stop condition $\max\limits_{(\widehat{X},\hat{x},\hat{\mu},\hat{\alpha})}w(\widehat{X},\hat{\alpha})\leq\varepsilon$ satisfies and $L_{NC}\neq \emptyset$ holds. These help to reduce the number of iterations and the CPU. Furthermore, we found an interesting phenomenon, that is, the number of iterations of Algorithm \ref{alg:Framwork} is one less than the number of globally optimal solutions.

From the above numerical experiments, for these instances with infinite number of optimal solutions or with high dimensional, the termination condition $\max\limits_{(\widehat{X},\hat{x},\hat{\mu},\hat{\alpha})}w(\widehat{X},\hat{\alpha})\leq\varepsilon$ is easier to satisfied than $L_{NC}=\emptyset$. This means that two termination conditions in the proposed algorithm can be used to reduce the number of iterations of the algorithm.


\section{Conclusions}\label{secremark}
An {\rm$\alpha$BB} convexification method based on box classification strategy is studied for non-convex single-objective optimization problems. The box classification strategy is proposed based on the convexity of the objective function on the sub-boxes when dividing the boxes, which helps to reduce the number of box divisions and improve the computational efficiency. The {\rm$\alpha$BB} method is introduced to construct the piecewise convexification problem of the non-convex optimization problem, and the solution set of the piecewise convexification problem is used to approximate the globally optimal solution set of the original problem. Based on the theoretical results, an {\rm$\alpha$BB} convexification algorithm with two termination conditions is proposed, and numerical experiments show that this algorithm can obtain a large number of globally optimal solutions more quickly than other algorithms.

\hbox to14cm{\hrulefill}\par
\noindent\textsc{Qiao Zhu}\\
{{\footnotesize
{College of Mathematics, Sichuan University, 610065, Chengdu Sichuan, China}\\
 {E-mail address: math\_qiaozhu@163.com}\\
{\scshape Liping Tang}\\
{\footnotesize
{National Center for Applied Mathematics, Chongqing Normal University, 401331 Chongqing, China.}\\
 { Email address: tanglipings@163.com}}
\medskip}\\
{\scshape Xinmin Yang}\\
{\footnotesize
{National Center for Applied Mathematics, Chongqing Normal University, 401331 Chongqing, China. }\\
 { Email: xmyang@cqnu.edu.cn}}\\

\end{document}